\documentclass[a4,10pt]{article}
\title{On some explicit evaluations of multiple zeta-star values}
\author{Shuichi Muneta}
\date{}

\usepackage{amsmath,amsthm,amssymb}

\begin{document}
\maketitle

\newtheorem{thm}{Theorem}
\newtheorem{cor}[thm]{Corollary}
\newtheorem{thmA}{Theorem}
\renewcommand{\thethmA}{\Alph{thmA}} 

\begin{abstract}
In this paper, we give some explicit evaluations of multiple zeta-star values 
which are rational multiple of powers of $\pi^2$.
\end{abstract}

\section{Main Results}
The multiple zeta value (MZV) is defined by the convergent series
\[\zeta(k_{1},k_{2},\ldots,k_{n})
  :=\sum_{m_{1}>m_{2}>\cdots>m_{n}>0}\frac{1}{m_{1}^{k_{1}}m_{2}^{k_{2}}{\cdots}m_{n}^{k_{n}}},\]
where $k_{1},k_{2},\ldots,k_{n}$ are positive integers and $k_{1}\geq 2$. 
The integers $k=k_{1}+k_{2}+\cdots+k_{n}$ and $n$ are called weight and depth respectively. 
Considerable amount of work on MZV's has been done in recent years from various aspects and interests. 
Among them, several explicit values are known for special index sets, as will be recalled below. 

In this paper, we give some evaluations of the \textit{multiple zeta-star value} (MZSV), which is 
defined by the following series similar to the MZV:
\[\zeta^*(k_{1},k_{2},\ldots,k_{n})
:=\sum_{m_{1} \geq m_{2} \geq \cdots \geq m_{n}>0} 
  \frac{1} {m_{1}^{k_{1}}m_{2}^{k_{2}} \cdots m_{n}^{k_{n}}},\]
where $k_{1},k_{2},\ldots,k_{n}$ satisfy the same condition as above. 
The MZSV can be expressed as a $\mathbb{Z}$-linear combination of MZV's, 
and vice versa. 

\begin{thmA}
For positive integers $m$,$n$, we have 
\begin{align*}
\lefteqn{\zeta^{*}(\underbrace{2m,2m,\cdots,2m}_{n})} \\
 &=\left \{\sum_{ \begin{subarray}{c} n_{0} + \cdots + n_{m-1} = mn \\ n_{i} \geq 0 \end{subarray} }  
   (-1)^{m(n-1)} \left(\prod _{k=0}^{m-1} \frac{(2^{2n_{k}}-2)B_{2n_{k}}} {(2n_{k})!} \right)
   \exp \left(\frac{2{\pi}i}{m} \sum_{l=0}^{m-1} ln_{l} \right) \right \} \pi^{2mn}.
\end{align*}
\label{thma}
\end{thmA}

\begin{thmA}
For positive integer $n$, we have
\begin{align*}
\lefteqn{\zeta^{*}(\underbrace{3,1,\cdots,3,1}_{2n})} \\
&=\sum_{i=0}^{n} \Bigg\{ \frac{2}{(4i+2)!} 
  \sum_{\begin{subarray}{c} n_{0} + n_{1} = 2(n-i) \\ n_{0}, n_{1} \geq 0 \end{subarray} } 
  (-1)^{n_{1}} \frac{ (2^{2n_{0}}-2) B_{2n_{0}} } {(2n_{0})!} 
  \frac{ (2^{2n_{1}}-2) B_{2n_{1}} } {(2n_{1})!} \Bigg\} {\pi}^{4n}.
\end{align*}
In particular,
\[\zeta^{*}(\underbrace{3,1,\cdots,3,1}_{2n}) \in \mathbb{Q} \times {\pi}^{4n}.\]
\label{thmb}
\end{thmA}

\begin{thmA}
Let $n$ be a positive integer, and let $I_{2n}$ denote the set of all $2n+1$ possible insertions 
of the number $2$ in the string 
$\{ \underbrace{3,1,\ldots,3,1}_{2n} \}$. Then we have 
\[\sum_{\vec{s}_{2n}\in I_{2n}} \! \zeta^{*} (\vec{s}_{2n}) 
=\sum_{k=0}^{n} \left\{ \frac{ 2^{4k+3} B_{4k+2} }{ (4k+2)! } 
 \sum_{i=0}^{n-k} \frac{ \alpha_{n-k-i} }{ (4i+2)! } 
 - \frac{ \alpha_{n-k} }{ (4k+3)! } \right\} \pi^{4n+2} ,\]
where
\begin{align*}
\alpha_{n} 
&=\sum_{\begin{subarray}{c} n_{0} + n_{1} = 2n \\ n_{0}, n_{1} \geq 0 \end{subarray} }
  (-1)^{n_{1}} \frac{ (2^{2n_{0}}-2) B_{2n_{0}} } {(2n_{0})!} 
               \frac{ (2^{2n_{1}}-2) B_{2n_{1}} } {(2n_{1})!}.
\end{align*}
In particular,
\[\sum_{\vec{s}_{2n}\in I_{2n}} \zeta^{*} (\vec{s}_{2n}) \in \mathbb{Q} \times \pi^{4n+2}.\]
\label{thmc}
\end{thmA}

For later use, we recall the corresponding results for MZV's.
\begin{thm}[{[AK]}]
Let m, n be positive integers. Then we have
\[\zeta(\underbrace{2m,2m,\ldots,2m}_{n})=C_{n}^{(m)} \frac{(2\pi i)^{2mn}}{(2mn)!},\]
where $C_{n}^{(m)}$ is defined by the following recurrence relations:
\[C_{0}^{(m)}=1, \; C_{n}^{(m)}=\frac{1}{2n} \sum_{l=1}^{n} (-1)^{l}
\binom{2mn}{2ml}B_{2ml}C_{n-l}^{(m)} \quad  (n\geq 1),\]
where $B_{2n}$ are the classical Bernoulli numbers.
\label{thm1}
\end{thm}

\begin{thm}[{[BBBL1]},{[BBBL2]}]
For any positive integer $n$, we have
\[\zeta(\underbrace{3,1,\ldots,3,1}_{2n})=\frac{2\pi^{4n}}{(4n+2)!}.\]
\label{thm2}
\end{thm}

\begin{thm}[{[BBBL1]}]
Let $n$ be a positive integer, and let $I_{2n}$ denote the set of all $2n+1$ possible insertions 
of the number $2$ in the string 
$\{ \underbrace{3,1,\ldots,3,1}_{2n} \}$. Then
\[\sum_{\vec{s}_{2n}\in I_{2n}} \zeta(\vec{s}_{2n})=\frac{\pi^{4n+2}}{(4n+3)!}.\]
\label{thm3}
\end{thm}

\section{Algebraic setup}
We use the algebraic setup of MZV's that was developed by Hoffman[H2]. 
Consider the non-commutative polynomial ring
\[\mathfrak{H}:=\mathbb{Q}\left\langle x,y \right\rangle\]
in two indeterminates $x,y$. We refer to monomials in $x$ and $y$ as words. We also define subrings
\[\mathfrak{H}^1 := \mathbb{Q} + \mathfrak{H} y\]
and
\[\mathfrak{H}^0 := \mathbb{Q} + x \mathfrak{H} y.\]
For an integer $k\geq 1$, put $z_{k} = x^{k-1}y$. 
Then the ring $\mathfrak{H}^1$ is freely generated by $z_{k} \, (k=1,2,3,\ldots)$. 
When $k \geq 2$, $z_{k}$ is contained $\mathfrak{H}^0$.  

Now define the evaluation map $Z:\mathfrak{H}^0 \longrightarrow \mathbb{R}$ by setting
\[Z(z_{k_{1}} z_{k_{2}} \cdots z_{k_{n}}) = \zeta(k_{1},k_{2},\ldots,k_{n})\]
on generators and extending it $\mathbb{Q}$-linearly. 

We define the harmonic product $*$ on $\mathfrak{H}^1$ inductively by
\begin{align*}
w*1 &= 1*w = w \\
z_{p}w_{1} * z_{q}w_{2} &= z_{p}(w_{1} * z_{q}w_{2}) + z_{q}(z_{p}w_{1} * w_{2}) 
                          + z_{p+q}(w_{1} *w_{2}),
\end{align*}
for all $p,q \geq 1$, and any words $w,w_{1},w_{2} \in \mathfrak{H}^1$, 
together with $\mathbb{Q}$-bilinearity. For instance, 
$z_{p} * z_{q}=z_{p}z_{q} + z_{q}z_{p} + z_{p+q}.$ 
This product corresponds to $\zeta(p)\zeta(q)=\zeta(p,q)+\zeta(q,p)+\zeta(p+q).$

The following theorem which has been proven in [H2] gives the basic algebraic properties 
of the $*$-product. 

\begin{thm}[{[H2]}]
The harmonic product is commutative and associative.
\label{h.p}
\end{thm}

Theorem\,\ref{h.p} says that $\mathfrak{H}^1$ is a $\mathbb{Q}$-commutative algebra 
with respect to the harmonic product $*$. Then $\mathfrak{H}^0$ is subalgebra of $\mathfrak{H}^1$. 
In [H2], it has also been proved that $Z$ is homomorphism with respect to the harmonic product $*$:
\[Z(w_{1} * w_{2}) = Z(w_{1}) Z(w_{2}).\quad  (w_{1},w_{2} \in \mathfrak{H}^0) \]
We conclude this section by introducing the $\mathbb{Q}$-linear map $S$. 
Let $S_{1} \in Aut(\mathfrak{H})$ be defined by $S_{1}(1)=1$, $S_{1}(x)=x$ and $S_{1}(y)=x+y$. 
Define the $\mathbb{Q}$-linear map $S$ : $\mathfrak{H}^{1} \longrightarrow \mathfrak{H}^{1}$ by
\[S(Fy):=S_{1}(F)y\]
for all words $F \in \mathfrak{H}$ and $S(1)=1$.  
Then it is clear that
\[\zeta^*(k_{1},k_{2},\ldots,k_{n}) = Z (S( z_{k_{1}} z_{k_{2}} \cdots z_{k_{n}} ) ).\]
For example, 
$\zeta^*(k_{1},k_{2}) = \zeta(k_{1} + k_{2}) + \zeta(k_{1} , k_{2} ) 
 = Z( S( z_{k_{1}} z_{k_{2}} ) )$. 
$\zeta^*(k_{1},k_{2},k_{3}) = \zeta(k_{1} + k_{2} + k_{3}) + \zeta(k_{1} + k_{2} , k_{3}) 
 + \zeta(k_{1} , k_{2} + k_{3}) + \zeta(k_{1} , k_{2} , k_{3}) 
 = Z( S( z_{k_{1}} z_{k_{2}} z_{k_{3}} ) )$.

\section{Proof of Theorem \ref{thma}}
We prove Theorem \ref{thma} by using the Laurent expansion for the cosecant function:
\[\csc x = \sum_{n=0}^{\infty}(-1)^{n-1} \frac{(2^{2n}-2)B_{2n}}{(2n)!}x^{2n-1}.\]
\begin{proof}[Proof of Theorem \ref{thma}]
Using the infinite product for the sine function, we have
\[\csc \pi x e^{\frac{\pi i} {m} k} 
 =\frac{1} {\displaystyle \pi x e^{\frac{\pi i} {m} k} \prod_{n=1}^{\infty}
 \left(1 - \frac{x^{2}} {n^{2}} e^{\frac{2\pi i} {m} k} \right)} .\]
Substituting $k=0,1,\ldots,m-1$ and multiplying both sides, we obtain
\begin{align}
\left(\pi^m x^{m} \prod_{k=0}^{m-1} e^{\frac{\pi i} {m}k} \right) 
\prod_{k=0}^{m-1}  \csc \pi x e^{\frac{\pi i} {m}k}
&=\frac{1} {\displaystyle \prod_{n=1}^{\infty} \left(1 - \frac{x^{2m}} {n^{2m}} \right)}.
\label{(1)}
\end{align}
The right hand side of \eqref{(1)} equals
\begin{align*}
& 1 + \left(\sum_{n_{1}>0} \frac{1} {n_{1}^{2m}} \right) x^{2m} 
  + \left(\sum_{n_{1} \geq n_{2} > 0} \frac{1} {n_{1}^{2m}n_{2}^{2m}} \right) x^{4m} + \cdots  
  \displaybreak[0]\\
&=  1 + \sum_{n=1}^{\infty} \zeta^{*}(\underbrace{2m,2m,\ldots,2m}_{n}) x^{2mn}.
\end{align*}
On the other hand, the left hand side of \eqref{(1)} equals
\begin{align*}
&\left(\pi^m x^{m} \prod_{k=0}^{m-1} e^{\frac{\pi i} {m}k} \right)
 \prod_{k=0}^{m-1} \sum_{n_{k}=0}^{\infty} (-1)^{n_{k}-1} 
 \frac{(2^{2n_{k}}-2)B_{2n_{k}}} {(2n_{k})!}
 \pi^{2n_{k}-1} x^{2n_{k}-1} e^{\frac{\pi i} {m} k(2n_{k}-1)}  \displaybreak[0]\\
&=\prod_{k=0}^{m-1} \sum_{n_{k}=0}^{\infty} (-1)^{n_{k}-1} 
 \frac{(2^{2n_{k}}-2)B_{2n_{k}}} {(2n_{k})!}
 \pi^{2n_{k}} x^{2n_{k}} e^{\frac{2 \pi i} {m} kn_{k}}  \displaybreak[0]\\
&=1+\sum_{n=1}^{\infty} 
  \left\{ \sum_{\begin{subarray}{c} n_{0} + \cdots + n_{m-1} = mn \\ n_{i} \geq 0 \end{subarray}}
  (-1)^{m(n-1)} \left(\prod _{k=0}^{m-1} 
  \frac{(2^{2n_{k}}-2)B_{2n_{k}}} {(2n_{k})!} \right) \right. \\
&\qquad\qquad\qquad\qquad\qquad\qquad\qquad\qquad\qquad  
 \left.\times \exp 
 \left(\frac{2{\pi}i} {m} \sum_{l=0}^{m-1} ln_{l} \right) \right \} \pi^{2mn} x^{2mn}. 
\end{align*}
Comparing coefficients of both sides, we obtain the desired identity.          
\end{proof}

\begin{cor}
For positive integers $m$,$n$, we have
\[\zeta^{*}(\underbrace{2m,2m,\ldots,2m}_{n}) \in \mathbb{Q} \times \pi^{2mn}.\]
\label{cor5}
\end{cor}
\begin{proof}
The coefficient of $\pi^{2mn}$ on the right hand side of Theorem \ref{thma} 
is invariant under the action of the galois group $Gal(\mathbb{Q}(\zeta_m)/\mathbb{Q})$, 
hence belongs to $\mathbb{Q}$.
\end{proof}

\textbf{Remark.}
Yasuo Ohno proves Theorem \ref{thma} independently. He proves this theorem in two ways, 
one way is to use the same method of our proof. 
The other is to use generating function and differential equation.

\section{Proof of Theorem \ref{thmb}}
Theorem \ref{thmb} will be obtained as a Corollary of a more general identity, 
which is stated as follows.

\begin{thm}
For positive integers $a$, $b$ and nonnegative integer $n$, we have
\begin{align}
 S( (z_{a} z_{b})^{n} ) 
 &= \sum_{i=0}^{n} (z_{a} z_{b})^{i} * S( z_{a+b}^{n-i} ) ,\label{eq:2} \\ 
 S( z_{b}(z_{a} z_{b})^{n} ) 
 &= \sum_{i=0}^{n} z_{b}(z_{a} z_{b})^{i} * S( z_{a+b}^{n-i} ). \label{eq:3}
\end{align}
\label{thm6}
\end{thm}
\begin{proof}
By definition of $S$, we have 
\[S(w_{1} w_{2}) = S_{1}(w_{1}) S(w_{2}). 
 \quad  (w_{1} \in \mathfrak{H}, w_{2} \in \mathfrak{H}^1) \]
Using this identity, we obtain 
\begin{align}
\lefteqn{ S( z_{k_{1}} z_{k_{2}} \cdots z_{k_{n}} ) } \nonumber\\
 &= z_{k_{1}} S( z_{k_{2}} z_{k_{3}} \cdots z_{k_{n}} ) 
  + S( z_{k_{1}+k_{2}} z_{k_{3}} \cdots z_{k_{n}} ) \nonumber\\
 &= z_{k_{1}} S( z_{k_{2}} z_{k_{3}} \cdots z_{k_{n}} ) 
  + z_{k_{1}+k_{2}} S( z_{k_{3}} \cdots z_{k_{n}} ) 
  + S( z_{k_{1}+k_{2}+k_{3}} z_{k_{4}} \cdots z_{k_{n}} ) \nonumber\\
 &= \cdots \nonumber\\
 &= \sum_{j=1}^{n} 
    z_{ k_{1} + k_{2} + \cdots + k_{j} } 
    S( z_{k_{j+1}} z_{k_{j+2}} \cdots z_{k_{n}}). \label{eq:4}
\end{align}
(When $j=n$, we regard $S( z_{k_{j+1}} z_{k_{j+2}} \cdots z_{k_{n}})$ as $1$.)
By using this identity, we obtain
\begin{align}
 S( (z_{a} z_{b})^{n} ) 
 &= \sum_{j=0}^{n-1} z_{(a+b)j+a} S( z_{b} (z_{a} z_{b})^{n-1-j} ) 
 +\sum_{j=1}^{n} z_{(a+b)j} S( (z_{a} z_{b})^{n-j} ), \label{eq:5} \displaybreak[0] \\
 S( z_{b}(z_{a} z_{b})^{n} ) 
 &= \sum_{j=0}^{n} z_{(a+b)j+b} S( (z_{a} z_{b})^{n-j} ) 
 +\sum_{j=1}^{n} z_{(a+b)j} S( z_{b} (z_{a} z_{b})^{n-j} ), \label{eq:6} \displaybreak[0] \\
 S( z_{a+b}^{n} ) &= \sum_{j=1}^{n} z_{(a+b)j} S( z_{a+b}^{n-j} ). \label{eq:7}
\end{align}
We prove identities (\ref{eq:2}) and (\ref{eq:3}) simultaneously by induction. 
The case of $n=0$ is obvious. Suppose that the assertion has been proven up to $n-1$. 
\begin{align*}
\lefteqn{\mathrm{RHS \; of \;} (\ref{eq:2})} \\
&\stackrel{(\ref{eq:7})}{=} 
   S( z_{a+b}^{n} ) 
   + \sum_{i=1}^{n-1} (z_{a} z_{b})^{i} * \sum_{j=1}^{n-i}z_{(a+b)j} S( z_{a+b}^{n-i-j} ) 
   + (z_{a} z_{b})^{n} \displaybreak[0]\\
&= S( z_{a+b}^{n} ) + \sum_{i=1}^{n-1} \sum_{j=1}^{n-i} z_{a} 
 \left( z_{b} (z_{a} z_{b})^{i-1} * z_{(a+b)j} S( z_{a+b}^{n-i-j} ) \right) \\
& \quad + \sum_{i=1}^{n-1} \sum_{j=1}^{n-i} z_{(a+b)j} 
 \left( (z_{a} z_{b})^{i} * S( z_{a+b}^{n-i-j} ) \right) \\
& \quad + \sum_{i=1}^{n-1} \sum_{j=1}^{n-i} z_{(a+b)j+a} 
   \left( z_{b} (z_{a} z_{b})^{i-1} * S( z_{a+b}^{n-i-j} ) \right) + (z_{a} z_{b})^{n} 
   \displaybreak[0]\\
&\stackrel{(\ref{eq:7})}{=}
  S( z_{a+b}^{n} ) + z_{a} \sum_{i=1}^{n-1} z_{b} (z_{a} z_{b})^{i-1} * S( z_{a+b}^{n-i} )  \\
& \quad + \sum_{j=1}^{n-1} z_{(a+b)j} \sum_{i=1}^{n-j} (z_{a} z_{b})^{i} * S( z_{a+b}^{n-j-i} ) \\
& \quad + \sum_{j=1}^{n-1} z_{(a+b)j+a} \sum_{i=1}^{n-j} 
  z_{b} (z_{a} z_{b})^{i-1} * S( z_{a+b}^{n-j-i} ) + (z_{a} z_{b})^{n} \displaybreak[0]\\
&= S( z_{a+b}^{n} ) + z_{a} 
 \left\{ S( z_{b} (z_{a} z_{b})^{n-1} ) - z_{b} (z_{a} z_{b})^{n-1} \right\} \\
& \quad + \sum_{j=1}^{n-1} z_{(a+b)j} 
 \left\{ S( (z_{a} z_{b} )^{n-j} ) - S( z_{a+b}^{n-j} ) \right\} \\
& \quad + \sum_{j=1}^{n-1} z_{(a+b)j+a} S(z_{b} (z_{a} z_{b})^{n-j-1} ) + (z_{a} z_{b})^{n} \\
& \qquad \mathrm{\big( by \; induction \; hypothesis \big) } \displaybreak[0]\\
&\stackrel{(\ref{eq:7})}{=}
  S( z_{a+b}^{n} ) + z_{a} S( z_{b} (z_{a} z_{b})^{n-1} ) 
  + \sum_{j=1}^{n-1} z_{(a+b)j} S( (z_{a} z_{b} )^{n-j} ) \\
& \quad - \left\{ S( z_{a+b}^{n} ) - z_{(a+b)n} \right\} 
  + \sum_{j=1}^{n-1} z_{(a+b)j+a} S(z_{b} (z_{a} z_{b})^{n-j-1} ) \displaybreak[0]\\
&= \sum_{j=0}^{n-1} z_{(a+b)j+a} S(z_{b} (z_{a} z_{b})^{n-j-1} ) 
  + \sum_{j=1}^{n} z_{(a+b)j} S( (z_{a} z_{b} )^{n-j} ) \displaybreak[0]\\
 &\stackrel{(\ref{eq:5})}{=} S( (z_{a} z_{b})^{n} ). 
 \end{align*} 
Hence, (\ref{eq:2}) is true for $n$. 
\begin{align*}
\lefteqn{\mathrm{RHS \; of \;} (\ref{eq:3})} \\
&\stackrel{(\ref{eq:7})}{=} 
 \sum_{i=0}^{n-1} z_{b}(z_{a} z_{b})^{i} * \sum_{j=1}^{n-i} z_{(a+b)j} S( z_{a+b}^{n-i-j} ) 
 + z_{b} (z_{a} z_{b})^{n} \displaybreak[0]\\
&= \sum_{i=0}^{n-1} \sum_{j=1}^{n-i} z_{b} 
 \left( (z_{a} z_{b})^{i} * z_{(a+b)j} S( z_{a+b}^{n-i-j}) \right) \\
&\quad + \sum_{i=0}^{n-1} \sum_{j=1}^{n-i} z_{(a+b)j} 
 \left( z_{b} (z_{a} z_{b})^{i} * S( z_{a+b}^{n-i-j})  \right) \\
& \quad + \sum_{i=0}^{n-1} \sum_{j=1}^{n-i} z_{(a+b)j+b} 
 \left( (z_{a} z_{b})^{i} * S( z_{a+b}^{n-i-j})  \right) 
 + z_{b} (z_{a} z_{b})^{n} \displaybreak[0]\\
&\stackrel{(\ref{eq:7})}{=} 
 z_{b} \sum_{i=0}^{n-1} (z_{a} z_{b})^{i} * S( z_{a+b}^{n-i})  \\
& \quad + \sum_{j=1}^{n} z_{(a+b)j} \sum_{i=0}^{n-j} 
 z_{b} (z_{a} z_{b})^{i} * S( z_{a+b}^{n-j-i}) \\
& \quad + \sum_{j=1}^{n} z_{(a+b)j+b} \sum_{i=0}^{n-j} 
 (z_{a} z_{b})^{i} * S( z_{a+b}^{n-j-i}) + z_{b} (z_{a} z_{b})^{n} \displaybreak[0]\\
&= z_{b} \left\{ S( (z_{a} z_{b})^{n} ) - (z_{a} z_{b})^{n}\right\} 
 + \sum_{j=1}^{n} z_{(a+b)j} S( z_{b} (z_{a} z_{b})^{n-j} ) \\
& \quad + \sum_{j=1}^{n} z_{(a+b)j+b} S( (z_{a} z_{b})^{n-j} ) + z_{b} (z_{a} z_{b})^{n} \\
& \qquad \mathrm{\big( by \; induction \; hypothesis \; and \; (\ref{eq:2}) \; for} \; n \big) 
  \displaybreak[0]\\
&= \sum_{j=0}^{n} z_{(a+b)j+b} S( (z_{a} z_{b})^{n-j} ) 
  + \sum_{j=1}^{n} z_{(a+b)j} S( z_{b} (z_{a} z_{b})^{n-j} )  \displaybreak[0]\\
&\stackrel{(\ref{eq:6})}{=} S( z_{b} (z_{a} z_{b})^{n} ). 
\end{align*}
Therefore, (\ref{eq:2}) is true for $n$.
\end{proof}

\begin{proof}[Proof of Theorem \ref{thmb}]
From (\ref{eq:2}), we have 
\[ \zeta^{*}(\underbrace{3,1,\cdots,3,1}_{2n}) 
 =\sum_{i=0}^{n} \zeta( \underbrace{3,1,\cdots,3,1}_{2i} ) 
  \zeta^{*}( \underbrace{4,4,\cdots,4}_{n-i} ).\]
Hence, we have the assertion by Theorem \ref{thm2} and Theorem \ref{thma}.
\end{proof}

\section{Proof of Theorem \ref{thmc}}
As in Section 4, we prove the following identities to obtain the explicit evaluations of 
$\sum_{\vec{s}_{2n} \in I_{2n}} \zeta^* ( \vec{s}_{2n} )$.

\begin{thm}
For positive integers $a$, $b$, $c$ and nonnegative integer $n$, we have
\begin{align}
\lefteqn{ 
 \sum_{k=0}^{n} S( (z_{a} z_{b})^{k} z_{c} (z_{a} z_{b})^{n-k} ) 
 + \sum_{k=0}^{n-1} S( z_{a} (z_{b} z_{a})^{k} z_{c} (z_{b} z_{a})^{n-1-k} z_{b} ) } \nonumber \\
&= 2 \sum_{k=0}^{n} z_{(a+b)k+c} * S( (z_{a} z_{b})^{n-k} ) \label{eq:8}\\
& \quad  - \sum_{k=0}^{n} S( z_{a+b}^{n-k} ) * 
 \left\{ \sum_{i=0}^{k} (z_{a} z_{b})^{i} z_{c} (z_{a} z_{b})^{k-i} 
 + \sum_{i=0}^{k-1} z_{a} (z_{b} z_{a})^{i} z_{c} (z_{b} z_{a})^{k-1-i} z_{b} \right\} \nonumber
\end{align}
and
\begin{align}
\lefteqn{ \sum_{k=0}^{n} S( z_{b} (z_{a} z_{b})^{k} z_{c} (z_{a} z_{b})^{n-k} ) 
          + \sum_{k=0}^{n} S( (z_{b} z_{a})^{k} z_{c} (z_{b} z_{a})^{n-k} z_{b} ) } \nonumber \\
&= 2 \sum_{k=0}^{n} z_{(a+b)k+c} * S( z_{b} (z_{a} z_{b})^{n-k} ) \label{eq:9}\\
&\quad  - \sum_{k=0}^{n} S( z_{a+b}^{n-k} ) * 
 \left\{ \sum_{i=0}^{k} z_{b} (z_{a} z_{b})^{i} z_{c} (z_{a} z_{b})^{k-i} 
 + \sum_{i=0}^{k} (z_{b} z_{a})^{i} z_{c} (z_{b} z_{a})^{k-i} z_{b} \right\}. \nonumber
\end{align}
$($We regard summations $\sum_{i=m}^{m-1}\cdots$ as $0$.$)$
\label{thm7}
\end{thm}
\begin{proof}
We put
\[A_{i,j} = (z_{a} z_{b})^i z_{c} (z_{a} z_{b})^j 
 \quad \mathrm{and} \quad
 B_{i,j} = (z_{b} z_{a})^{i} z_{c} (z_{b} z_{a})^j z_{b}.\]
Then we can rewrite (\ref{eq:8}) and (\ref{eq:9}) as follows:
\begin{align}
 \lefteqn{ 
 \sum_{k=0}^{n} S( A_{k,n-k} ) + \sum_{k=0}^{n-1} S( z_{a} B_{k,n-1-k} ) 
 } 
 \nonumber \\
&= 2 \sum_{k=0}^{n} z_{(a+b)k+c} * S( (z_{a} z_{b})^{n-k} ) \nonumber \\
& \qquad\qquad -
  \sum_{k=0}^{n} S( z_{a+b}^{n-k} ) * 
  \left\{ 
  \sum_{i=0}^{k} A_{i,k-i} + \sum_{i=0}^{k-1} z_{a} B_{i,k-1-i} 
  \right\} \label{eq:10}
\end{align}
and
\begin{align}
 \lefteqn{ 
 \sum_{k=0}^{n} S( z_{b} A_{k,n-k} ) + \sum_{k=0}^{n} S( B_{k,n-k} ) 
 } 
 \nonumber \\
&= 2 \sum_{k=0}^{n} z_{(a+b)k+c} * S( z_{b} (z_{a} z_{b})^{n-k} ) \nonumber\\
& \qquad\qquad - 
 \sum_{k=0}^{n} S( z_{a+b}^{n-k} ) * 
 \left\{ 
 \sum_{i=0}^{k} z_{b} A_{i,k-i} + \sum_{i=0}^{k} B_{i,k-i} 
 \right\}. \label{eq:11}
\end{align}
We prove the identities (\ref{eq:10}) and (\ref{eq:11}) simultaneously by induction. 
Before proceeding the proof, by using equation (\ref{eq:4}), 
we rewrite the quantities on the LHSs of 
(\ref{eq:10}) and (\ref{eq:11}). 
\begin{align*}
\lefteqn{S( A_{k,n-k} )} \\
 &=\sum_{j=1}^{k} 
   \left\{ 
   z_{(a+b)(j-1)+a} S( z_{b} A_{k-j,n-k} ) + z_{(a+b)j} S( A_{k-j,n-k} ) 
   \right\} \\
 & \quad + z_{(a+b)k+c} S( (z_{a} z_{b})^{n-k} ) \\
 & \quad + \sum_{j=k+1}^{n} 
  \left\{ 
  z_{(a+b)(j-1)+a+c} S( z_{b} (z_{a} z_{b})^{n-j} )
  + z_{(a+b)j+c} S( (z_{a} z_{b})^{n-j} ) 
  \right\}
\end{align*}
for $0 \leq k \leq n$, 
\begin{align*}
\lefteqn{
S( z_{a} B_{k,n-1-k} ) 
} \\
 &=\sum_{j=1}^{k} 
   \left\{ 
   z_{(a+b)(j-1)+a} S( B_{k-j+1,n-1-k} ) + z_{(a+b)j} S( z_{a} B_{k-j,n-1-k} ) 
   \right\} \\
 & \quad + z_{(a+b)k+a} S( B_{0,n-1-k}) \\
 & \quad + \sum_{j=k+1}^{n} 
  \left\{ 
  z_{(a+b)(j-1)+a+c} S( (z_{b} z_{a})^{n-j} z_{b}) 
   +z_{(a+b)j+c} S( (z_{a} z_{b})^{n-j}) 
   \right\}
\end{align*}
for $0 \leq k \leq n-1$,
\begin{align*}
\lefteqn{
S( z_{b} A_{k,n-k} )
} \\
 &=\sum_{j=1}^{k} 
   \left\{ 
   z_{(a+b)(j-1)+b} S( A_{k-j+1,n-k} ) + z_{(a+b)j} S( z_{b} A_{k-j,n-k} ) 
   \right\} \\
 & \quad + z_{(a+b)k+b} S (A_{0,n-k} )  \\
 & \quad + \sum_{j=k+1}^{n} 
 \left\{ 
 z_{(a+b)(j-1)+b+c} S( (z_{a} z_{b})^{n-j+1} ) + z_{(a+b)j+c} S( z_{b} (z_{a} z_{b})^{n-j} ) 
 \right\} \\
 & \quad + z_{(a+b)n+b+c}
\end{align*}
for $0 \leq k \leq n$ and
\begin{align*}
\lefteqn{S( B_{k,n-k} )} \\
 &=\sum_{j=1}^{k} 
   \left\{ 
   z_{(a+b)(j-1)+b} S( z_{a} B_{k-j,n-k} ) + z_{(a+b)j} S( B_{k-j,n-k} )  
   \right\} \\
 & \quad + z_{(a+b)k+c} S ( (z_{b} z_{a})^{n-k} z_{b} )  
   \displaybreak[0]\\
 & \quad + 
 \sum_{j=k+1}^{n} 
 \left\{ 
 z_{(a+b)(j-1)+b+c} S( (z_{a} z_{b})^{n-j+1} ) 
 + z_{(a+b)j+c} S( (z_{b} z_{a})^{n-j} z_{b}) 
 \right\} 
 \displaybreak[0]\\
 & \quad 
 + z_{(a+b)n+b+c}
\end{align*}
for $0 \leq k \leq n$. Hence, we have 
\begin{align}
 \lefteqn{
 \sum_{k=0}^{n} S( A_{k,n-k} ) + \sum_{k=0}^{n-1} S( z_{a} B_{k,n-1-k} ) 
 } 
 \nonumber \\
&= \sum_{k=0}^{n} z_{(a+b)k+c} S( (z_{a} z_{b})^{n-k} ) \nonumber \\
&\quad 
 + 2 \sum_{k=0}^{n-1} \sum_{j=k+1}^{n} 
 \left\{ 
 z_{(a+b)(j-1)+a+c} S( z_{b} (z_{a} z_{b})^{n-j} ) 
 + z_{(a+b)j+c} S( (z_{a} z_{b})^{n-j} ) 
 \right\} 
 \nonumber \displaybreak[0]\\
&\quad + 
  \sum_{j=1}^{n} z_{(a+b)j} 
  \left\{ 
  \sum_{k=0}^{n-j} S( A_{k,n-j-k} ) + \sum_{k=0}^{n-j-1} S(z_{a} B_{k,n-j-1-k} ) 
  \right\} 
  \nonumber\\
&\quad + 
  \sum_{j=1}^{n} z_{(a+b)(j-1)+a} 
  \left\{ 
  \sum_{k=0}^{n-j} S( z_{b} A_{k,n-j-k} ) + \sum_{k=0}^{n-j} S( B_{k,n-j-k} ) 
  \right\} \label{eq:12}
\end{align}
and
\begin{align}
 \lefteqn{ 
 \sum_{k=0}^{n} S( z_{b} A_{k,n-k} ) + \sum_{k=0}^{n} S( B_{k,n-k} ) 
 } 
 \nonumber \\
&= \sum_{k=0}^{n} z_{(a+b)k+c} S( z_{b} (z_{a} z_{b})^{n-k}) 
 \nonumber \displaybreak[0]\\
&\quad 
 + 2\sum_{k=0}^{n-1} \sum_{j=0}^{n-k} z_{(a+b)(k+j)+b+c} S( (z_{a} z_{b})^{n-j-k} ) 
 \nonumber \displaybreak[0]\\
&\quad 
 + 2\sum_{k=0}^{n-1} \sum_{j=1}^{n-k} 
 z_{(a+b)(k+j)+c} S( z_{b} (z_{a} z_{b})^{n-j-k} ) + 2z_{(a+b)n+b+c} 
 \nonumber \displaybreak[0]\\
&\quad + 
 \sum_{j=1}^{n} z_{(a+b)j} 
 \left\{ 
 \sum_{k=0}^{n-j} S( z_{b} A_{k,n-j-k} ) + \sum_{k=0}^{n-j} S( B_{k,n-j-k} ) 
 \right\} \nonumber \\
&\quad + 
 \sum_{j=0}^{n} z_{(a+b)j+b} 
 \left\{ 
 \sum_{k=0}^{n-j} S( A_{k,n-j-k} ) + \sum_{k=0}^{n-j-1} S( z_{a} B_{k,n-j-1-k} ) 
 \right\}. \label{eq:13}
\end{align} 

Now, the case of $n=0$ is obvious. Suppose that the assertion has been proven up to $n-1$.
\begin{align*}
\lefteqn{\sum_{k=0}^{n} z_{(a+b)k+c} * S( (z_{a} z_{b})^{n-k} )} \\
 &\stackrel{(\ref{eq:5})}{=} 
  \sum_{k=0}^{n-1} z_{(a+b)k+c} 
  * \left\{ \sum_{j=1}^{n-k} 
  z_{(a+b)(j-1)+a} S( z_{b} (z_{a} z_{b})^{n-k-j}) \right. \displaybreak[0]\\
 & \left. \qquad \qquad \qquad \qquad \qquad 
  + \sum_{j=1}^{n-k} 
  z_{(a+b)j} S( (z_{a} z_{b})^{n-k-j}) \right\} + z_{(a+b)n+c} \displaybreak[0]\\
 &= \sum_{k=0}^{n-1} \sum_{j=1}^{n-k} 
  z_{(a+b)k+c} z_{(a+b)(j-1)+a} S( z_{b} (z_{a} z_{b})^{n-k-j}) \displaybreak[0]\\
 &\quad + \sum_{k=0}^{n-1} \sum_{j=1}^{n-k} 
  z_{(a+b)(j-1)+a} \left( z_{(a+b)k+c} * S( z_{b} (z_{a} z_{b})^{n-k-j} ) \right) 
  \displaybreak[0]\\
 &\quad + \sum_{k=0}^{n-1} \sum_{j=1}^{n-k} 
  z_{(a+b)(k+j-1)+a+c} S( z_{b} (z_{a} z_{b})^{n-k-j} ) \displaybreak[0]\\
 &\quad + \sum_{k=0}^{n-1} \sum_{j=1}^{n-k} 
  z_{(a+b)k+c} z_{(a+b)j} S( (z_{a} z_{b})^{n-k-j} ) \displaybreak[0]\\
 &\quad + \sum_{k=0}^{n-1} \sum_{j=1}^{n-k} 
  z_{(a+b)j} \left( z_{(a+b)k+c} * S( (z_{a} z_{b})^{n-k-j} ) \right) \displaybreak[0]\\
 &\quad + \sum_{k=0}^{n-1} \sum_{j=1}^{n-k} 
  z_{(a+b)(k+j)+c} S( (z_{a} z_{b})^{n-k-j} ) + z_{(a+b)n+c} \displaybreak[0]\\
 &\stackrel{(\ref{eq:5})}{=}
  \sum_{k=0}^{n} z_{(a+b)k+c} S( (z_{a} z_{b})^{n-k} ) \displaybreak[0]\\
 &\quad + \sum_{j=1}^{n} z_{(a+b)(j-1)+a} 
  \sum_{k=0}^{n-j} z_{(a+b)k+c} * S( z_{b} (z_{a} z_{b})^{n-k-j} ) \displaybreak[0]\\
 &\quad + \sum_{j=1}^{n} z_{(a+b)j} 
  \sum_{k=0}^{n-j} z_{(a+b)k+c} * S( (z_{a} z_{b})^{n-k-j} ) \displaybreak[0]\\
 &\quad + \sum_{k=0}^{n-1} \sum_{j=1}^{n-k} 
  \left\{ z_{(a+b)(k+j-1)+a+c} S( z_{b} (z_{a} z_{b})^{n-k-j}) 
        + z_{(a+b)(k+j)+c} S( (z_{a} z_{b})^{n-k-j}) \right\}.
\end{align*}
On the other hand, 
\begin{align*}
 \lefteqn{ 
 \sum_{k=0}^{n} S( z_{a+b}^{n-k} ) * 
 \left\{ \sum_{i=0}^{k} A_{i,k-i} + \sum_{i=0}^{k-1} z_{a} B_{i,k-1-i} \right\} 
 } \\
 &= S(z_{a+b}^{n}) * z_{c} \displaybreak[0] \\
 &\quad + \sum_{k=1}^{n-1} S( z_{a+b}^{n-k} ) * 
  \left\{ \sum_{i=1}^{k} A_{i,k-i} + \sum_{i=0}^{k-1} z_{a} B_{i,k-1-i} \right\}  
  + \sum_{k=1}^{n-1} S(z_{a+b}^{n-k}) * z_{c} (z_{a} z_{b})^{k} \\
 &\quad + \sum_{i=0}^{n} A_{i,n-i} + \sum_{i=0}^{n-1} z_{a} B_{i,n-1-i} \displaybreak[0] \\
 &\stackrel{(\ref{eq:7})}{=} 
  \sum_{k=1}^{n-1} \sum_{j=1}^{n-k} z_{(a+b)j} S( z_{a+b}^{n-k-j}) * 
  \left\{ \sum_{i=1}^{k} A_{i,k-i}+ \sum_{i=0}^{k-1} z_{a} B_{i,k-1-i} \right\} 
  \displaybreak[0] \\
 &\quad + \sum_{k=0}^{n-1} \sum_{j=1}^{n-k} 
  z_{(a+b)j} S( z_{a+b}^{n-k-j} ) * z_{c}  (z_{a} z_{b})^{k} \displaybreak[0] \\
 &\quad + \sum_{i=0}^{n} A_{i,n-i} + \sum_{i=0}^{n-1} z_{a} B_{i,n-1-i} \displaybreak[0]  \\
 &=\sum_{k=1}^{n-1} \sum_{j=1}^{n-k} z_{(a+b)j} 
  \left\{ 
  S( z_{a+b}^{n-k-j} ) * 
  \left( \sum_{i=1}^{k} A_{i,k-i} + \sum_{i=0}^{k-1} z_{a} B_{i,k-1-i} \right) 
  \right\} \displaybreak[0] \\
 &\quad+ \sum_{k=1}^{n-1} \sum_{j=1}^{n-k} z_{a} 
  \left\{ 
  z_{(a+b)j} S( z_{a+b}^{n-k-j}) * 
  \left( \sum_{i=1}^{k} z_{b} A_{i-1,k-i} + \sum_{i=0}^{k-1} B_{i,k-1-i} \right) 
  \right\} \displaybreak[0] \\
 &\quad+ \sum_{k=1}^{n-1} \sum_{j=1}^{n-k} z_{(a+b)j+a} 
  \left\{ 
  S( z_{a+b}^{n-k-j}) * 
  \left( \sum_{i=1}^{k} z_{b} A_{i-1,k-i} + \sum_{i=0}^{k-1} B_{i,k-1-i} \right) 
  \right\} \displaybreak[0] \\
 &\quad +\sum_{k=0}^{n-1} \sum_{j=1}^{n-k} 
  \left\{ 
  z_{(a+b)j} \left( S( z_{a+b}^{n-k-j} ) * z_{c}  (z_{a} z_{b})^{k} \right)   \right. \\
 & \left.\quad + z_{c} \left( z_{(a+b)j} S( z_{a+b}^{n-k-j} ) * (z_{a} z_{b})^{k} \right)
  + z_{(a+b)j+c} \left( S( z_{a+b}^{n-k-j} ) * (z_{a} z_{b})^{k} \right) 
  \right\} \\
 &\quad+ \sum_{i=0}^{n} A_{i,n-i} + \sum_{i=0}^{n-1} z_{a} B_{i,n-1-i} \displaybreak[0]\\
 &=\sum_{k=1}^{n-1} \sum_{j=1}^{n-k} z_{(a+b)j} 
  \left\{
  S( z_{a+b}^{n-k-j}) *  
  \left( \sum_{i=1}^{k} A_{i,k-i} + \sum_{i=0}^{k-1} z_{a} B_{i,k-1-i} \right) 
  \right\} \\
 &\quad+
  \sum_{k=0}^{n-1} \sum_{j=1}^{n-k} z_{(a+b)j} 
  \left( S( z_{a+b}^{n-k-j} ) * z_{c}  (z_{a} z_{b})^{k} \right) \displaybreak[0] \\
 &\quad+ \sum_{k=0}^{n-1} \sum_{j=1}^{n-k} 
  \left\{
  z_{c} \left( z_{(a+b)j} S( z_{a+b}^{n-k-j} ) * (z_{a} z_{b})^{k} \right) 
  + z_{(a+b)j+c} \left( S( z_{a+b}^{n-k-j} ) * (z_{a} z_{b})^{k} \right)  
  \right\} \displaybreak[0]\\
 &\quad+ \sum_{k=1}^{n-1} \sum_{j=1}^{n-k} z_{a} 
  \left\{ 
  z_{(a+b)j} S( z_{a+b}^{n-k-j}) * 
  \left( \sum_{i=1}^{k} z_{b} A_{i-1,k-i} + \sum_{i=0}^{k-1} B_{i,k-1-i} \right) 
  \right\} \\
 &\quad+ \sum_{k=1}^{n-1} \sum_{j=1}^{n-k} z_{(a+b)j+a} 
  \left\{ 
  S( z_{a+b}^{n-k-j}) * 
  \left( \sum_{i=1}^{k} z_{b} A_{i-1,k-i} + \sum_{i=0}^{k-1} B_{i,k-1-i} \right) 
  \right\} 
  \displaybreak[0] \\
 &\quad+ \sum_{i=0}^{n} A_{i,n-i} + \sum_{i=0}^{n-1} z_{a} B_{i,n-1-i} \displaybreak[0] \\
 &\stackrel{(\ref{eq:7})}{=}
  \sum_{k=1}^{n-1} \sum_{j=1}^{n-k} z_{(a+b)j} 
  \left\{ 
  S( z_{a+b}^{n-k-j}) * 
  \left( \sum_{i=0}^{k} A_{i,k-i} + \sum_{i=0}^{k-1} z_{a} B_{i,k-1-i} \right) 
 \right\} 
  \displaybreak[0] \\
 &\quad + \sum_{j=1}^{n} z_{(a+b)j} 
  \left( S( z_{a+b}^{n-j} ) * z_{c} \right) \displaybreak[0] \\
 &\quad+ z_{c} \sum_{k=0}^{n-1} S( z_{a+b}^{n-k} ) * (z_{a} z_{b})^{k}  
  + \sum_{j=1}^{n} z_{(a+b)j+c} \sum_{k=0}^{n-j} 
  S( z_{a+b}^{n-k-j} ) * (z_{a} z_{b})^{k} \displaybreak[0]\\
 &\quad+ \sum_{k=1}^{n-1} \sum_{j=0}^{n-k} z_{(a+b)j+a} 
  \left\{ 
  S( z_{a+b}^{n-k-j}) * 
  \left( \sum_{i=1}^{k} z_{b} A_{i-1,k-i} + \sum_{i=0}^{k-1} B_{i,k-1-i} \right) 
  \right\} 
  \displaybreak[0] \\
 &\quad+ \sum_{i=0}^{n} A_{i,n-i} + \sum_{i=0}^{n-1} z_{a} B_{i,n-1-i} \displaybreak[0]\\
 &\stackrel{(\ref{eq:2})}{=}
  \sum_{k=0}^{n-1} \sum_{j=1}^{n-k} z_{(a+b)j} 
  \left\{ 
  S( z_{a+b}^{n-k-j}) *  
  \left( \sum_{i=0}^{k} A_{i,k-i} + \sum_{i=0}^{k-1} z_{a} B_{i,k-1-i} \right) 
  \right\} 
  \displaybreak[0] \\ 
 &\quad+ z_{c} S( (z_{a} z_{b})^{n} ) - z_{c} (z_{a} z_{b})^{n} 
  + \sum_{j=1}^{n} z_{(a+b)j+c} S( (z_{a} z_{b})^{n-j} ) \displaybreak[0]\\
 &\quad+ \sum_{k=1}^{n-1} \sum_{j=1}^{n+1-k} z_{(a+b)(j-1)+a} 
  \left\{ 
  S( z_{a+b}^{n-k-j+1}) * 
  \left( \sum_{i=0}^{k-1} z_{b} A_{i,k-1-i} + \sum_{i=0}^{k-1} B_{i,k-1-i} \right) 
  \right\} 
  \displaybreak[0] \\
 &\quad+ \sum_{i=0}^{n} A_{i,n-i} + \sum_{i=0}^{n-1} z_{a} B_{i,n-1-i} \displaybreak[0]\\
 &= \sum_{k=0}^{n-1} \sum_{j=1}^{n-k} z_{(a+b)j} 
  \left\{ 
  S( z_{a+b}^{n-k-j}) * 
  \left( \sum_{i=0}^{k} A_{i,k-i} + \sum_{i=0}^{k-1} z_{a} B_{i,k-1-i} \right) 
  \right\} 
  \displaybreak[0] \\
 &\quad+ \sum_{j=0}^{n} z_{(a+b)j+c} S( (z_{a} z_{b})^{n-j} )   \\
 &\quad+ \sum_{k=0}^{n-2} \sum_{j=1}^{n-k} z_{(a+b)(j-1)+a} 
  \left\{ 
  S( z_{a+b}^{n-k-j}) * 
  \left( \sum_{i=0}^{k} z_{b} A_{i,k-i} + \sum_{i=0}^{k} B_{i,k-i} \right) 
  \right\}  
  \displaybreak[0]\\
 &\quad+ 
  \sum_{i=1}^{n} A_{i,n-i} + \sum_{i=0}^{n-1} z_{a} B_{i,n-1-i} \displaybreak[0] \\ 
 &= \sum_{j=0}^{n} z_{(a+b)j+c} S( (z_{a} z_{b})^{n-j} ) 
  \displaybreak[0]   \\
 &\quad+ \sum_{k=0}^{n-1} \sum_{j=1}^{n-k} z_{(a+b)j} 
  \left\{ 
  S( z_{a+b}^{n-k-j}) * 
  \left( \sum_{i=0}^{k} A_{i,k-i} + \sum_{i=0}^{k-1} z_{a} B_{i,k-1-i} \right) 
  \right\} 
  \displaybreak[0] \\
 &\quad+ \sum_{k=0}^{n-1} \sum_{j=1}^{n-k} z_{(a+b)(j-1)+a} 
  \left\{ 
  S( z_{a+b}^{n-k-j}) * 
  \left( \sum_{i=0}^{k} z_{b} A_{i,k-i} + \sum_{i=0}^{k} B_{i,k-i} \right) 
  \right\}.
\end{align*}

Therefore, we have
\begin{align*}
 \lefteqn{ 2 \sum_{k=0}^{n} z_{(a+b)k+c} * S( (z_{a} z_{b})^{n-k} ) } 
 \\
&\quad\quad
 - \sum_{k=0}^{n} S( z_{a+b}^{n-k} ) * 
 \left\{ \sum_{i=0}^{k} A_{i,k-i} + \sum_{i=0}^{k-1} z_{a} B_{i,k-1-i} \right\}   
 \\
&= \sum_{k=0}^{n} z_{(a+b)k+c} S( (z_{a} z_{b})^{n-k} ) \displaybreak[0]  \\
&\quad + 2\sum_{k=0}^{n-1} \sum_{j=1}^{n-k} 
 \left\{ z_{(a+b)(k+j-1)+a+c} S( z_{b} (z_{a} z_{b})^{n-k-j}) 
 + z_{(a+b)(k+j)+c} S( (z_{a} z_{b})^{n-k-j}) \right\} \displaybreak[0]  \\
&\quad + \sum_{j=1}^{n} z_{(a+b)j} 
  \left\{ 
  2 \sum_{k=0}^{n-j} z_{(a+b)k+c} * S( (z_{a} z_{b})^{n-j-k} ) \right.\\
&\left. \qquad\qquad\qquad\qquad
  - \sum_{k=0}^{n-j} S( z_{a+b}^{n-j-k} ) * 
  \left( \sum_{i=0}^{k} A_{i,k-i} + \sum_{i=0}^{k-1} z_{a} B_{i,k-1-i} \right) 
  \right\} 
  \displaybreak[0]  \\
 &\quad +\sum_{j=1}^{n} z_{(a+b)(j-1)+a} 
  \left\{ 
  2 \sum_{k=0}^{n-j} z_{(a+b)k+c} * S( z_{b} (z_{a} z_{b})^{n-j-k} ) \right. \\
 &\left. \qquad\qquad\qquad\qquad 
  - \sum_{k=0}^{n-j} S( z_{a+b}^{n-j-k} ) * 
  \left( \sum_{i=0}^{k} z_{b} A_{i,k-i} + \sum_{i=0}^{k} B_{i,k-i} \right) 
  \right\} 
  \displaybreak[0]  \\
 &= \sum_{k=0}^{n} z_{(a+b)k+c} S( (z_{a} z_{b})^{n-k} )  \displaybreak[0]\\
 &\quad + 2 \sum_{k=0}^{n-1} \sum_{j=k+1}^{n} 
  \left\{ z_{(a+b)(j-1)+a+c} S( z_{b} (z_{a} z_{b})^{n-j} )
  + z_{(a+b)j+c} S( (z_{a} z_{b})^{n-j} ) \right\} 
  \displaybreak[0]\\
 &\quad + 
  \sum_{j=1}^{n} z_{(a+b)j} 
  \left\{ \sum_{k=0}^{n-j} S( A_{k,n-j-k}) 
  + \sum_{k=0}^{n-j-1} S(z_{a} B_{k,n-j-1-k} ) \right\} \\
 &\quad + 
  \sum_{j=1}^{n} z_{(a+b)(j-1)+a} 
  \left\{ \sum_{k=0}^{n-j} S( z_{b} A_{k,n-j-k} ) 
  + \sum_{k=0}^{n-j} S( B_{k,n-j-k} ) \right\} \\
 & \qquad \mathrm{\big( by \; induction \; hypothesis \big) }
  \displaybreak[0]\\
 &\stackrel{(\ref{eq:12})}{=}
  \sum_{k=0}^{n} S( A_{k,n-k} ) + \sum_{k=0}^{n-1} S( z_{a} B_{k,n-1-k} ). 
\end{align*}
Hence (\ref{eq:10}) is true for $n$. Next we prove (\ref{eq:11}) for $n$.
\begin{align*}
\lefteqn{ \sum_{k=0}^{n} z_{(a+b)k+c} * S( z_{b} (z_{a} z_{b})^{n-k} )} \\
 &\stackrel{(\ref{eq:6})}{=}  
  \sum_{k=0}^{n-1} z_{(a+b)k+c} * 
  \left\{ \sum_{j=0}^{n-k} z_{(a+b)j+b} S( (z_{a} z_{b})^{n-k-j} ) 
  +\sum_{j=1}^{n-k} z_{(a+b)j} S( z_{b} (z_{a} z_{b})^{n-k-j} )  \right\} \displaybreak[0]\\
 &\quad + z_{(a+b)n+c} * z_{b} \displaybreak[0]\\
 &= \sum_{k=0}^{n-1} \sum_{j=0}^{n-k} 
   z_{(a+b)k+c} z_{(a+b)j+b} S( (z_{a} z_{b})^{n-k-j} ) \displaybreak[0]\\
 &\quad + \sum_{k=0}^{n-1} \sum_{j=0}^{n-k} z_{(a+b)j+b} 
  \left( z_{(a+b)k+c} * S( (z_{a} z_{b})^{n-k-j} ) \right) \displaybreak[0]\\
 &\quad + \sum_{k=0}^{n-1} \sum_{j=0}^{n-k} 
  z_{(a+b)(k+j)+b+c} S( (z_{a} z_{b})^{n-k-j} ) \displaybreak[0]\\
 &\quad + \sum_{k=0}^{n-1} \sum_{j=1}^{n-k} 
  z_{(a+b)k+c} z_{(a+b)j} S( z_{b} (z_{a} z_{b})^{n-k-j} ) \displaybreak[0]\\
 &\quad + \sum_{k=0}^{n-1} \sum_{j=1}^{n-k} z_{(a+b)j} 
  \left( z_{(a+b)k+c} * S( z_{b} (z_{a} z_{b})^{n-k-j} ) \right) \displaybreak[0]\\
 &\quad + \sum_{k=0}^{n-1} \sum_{j=1}^{n-k} 
  z_{(a+b)(k+j)+c} S( z_{b} (z_{a} z_{b})^{n-k-j} ) \displaybreak[0]\\
 &\quad + z_{(a+b)n+c} z_{b} + z_{b} z_{(a+b)n+c} + z_{(a+b)n+b+c} \displaybreak[0]\\
 &\stackrel{(\ref{eq:6})}{=}
  \sum_{k=0}^{n} z_{(a+b)k+c} S( z_{b} (z_{a} z_{b})^{n-k} ) \displaybreak[0]\\
 &\quad + \sum_{j=0}^{n} z_{(a+b)j+b} \sum_{k=0}^{n-j} 
  z_{(a+b)k+c} * S( (z_{a} z_{b})^{n-k-j} ) \displaybreak[0]\\
 &\quad + \sum_{j=1}^{n} z_{(a+b)j} \sum_{k=0}^{n-j} 
  z_{(a+b)k+c} * S( z_{b} (z_{a} z_{b})^{n-k-j} ) \displaybreak[0]\\
 &\quad + \sum_{k=0}^{n-1} \sum_{j=0}^{n-k} 
  z_{(a+b)(k+j)+b+c} S( (z_{a} z_{b})^{n-k-j} ) \displaybreak[0]\\
 &\quad + \sum_{k=0}^{n-1} \sum_{j=1}^{n-k} 
  z_{(a+b)(k+j)+c} S( z_{b} (z_{a} z_{b})^{n-k-j} ) \displaybreak[0]\\
 &\quad + z_{(a+b)n+b+c}.
\end{align*}
On the other hand, 
\begin{align*}
 \lefteqn{
 \sum_{k=0}^{n} S( z_{a+b}^{n-k} ) * 
 \left\{ \sum_{i=0}^{k} z_{b} A_{i,k-i} + \sum_{i=0}^{k} B_{i,k-i} \right\}
 } 
 \displaybreak[0]\\
 &= S( z_{a+b}^n ) * ( z_{b} z_{c} + z_{c} z_{b} )  \displaybreak[0]\\
 &\quad + 
  \sum_{k=1}^{n-1} S( z_{a+b}^{n-k} ) * 
  \left\{ \sum_{i=0}^{k} z_{b} A_{i,k-i} + \sum_{i=1}^{k} B_{i,k-i} \right\} 
  + \sum_{k=1}^{n-1} S( z_{a+b}^{n-k} ) * z_{c} (z_{b} z_{a})^{k} z_{b} \displaybreak[0]\\
 &\quad + 
  \sum_{i=0}^{n} z_{b} A_{i,n-i} + \sum_{i=0}^{n} B_{i,n-i} 
  \displaybreak[0]\\
 &\stackrel{(\ref{eq:7})}{=}
  \sum_{k=1}^{n-1} \sum_{j=1}^{n-k} z_{(a+b)j} S( z_{a+b}^{n-k-j} ) * 
  \left\{ \sum_{i=0}^{k} z_{b} A_{i,k-i} + \sum_{i=1}^{k} B_{i,k-i} \right\} 
  \displaybreak[0]\\
 &\quad + \sum_{k=0}^{n-1} \sum_{j=1}^{n-k} z_{(a+b)j} 
  S( z_{a+b}^{n-k-j}) * z_{c} (z_{b} z_{a})^{k} z_{b}
  + \sum_{j=1}^{n} z_{(a+b)j} S( z_{a+b}^{n-j}) * z_{b} z_{c} \displaybreak[0]\\
 &\quad + 
  \sum_{i=0}^{n} z_{b} A_{i,n-i}
  + \sum_{i=0}^{n} B_{i,n-i} 
  \displaybreak[0]\\
 &= \sum_{k=1}^{n-1} \sum_{j=1}^{n-k} z_{(a+b)j} 
  \left\{ 
  S( z_{a+b}^{n-k-j} ) * 
  \left( \sum_{i=0}^{k} z_{b} A_{i,k-i} + \sum_{i=1}^{k} B_{i,k-i} \right) 
  \right\} 
  \displaybreak[0]\\
 &\quad + 
  \sum_{k=1}^{n-1} \sum_{j=1}^{n-k} z_{b} 
  \left\{ 
  z_{(a+b)j} S( z_{a+b}^{n-k-j} ) 
  \left( \sum_{i=0}^{k} A_{i,k-i} + \sum_{i=1}^{k} z_{a} B_{i-1,k-i} \right)  
  \right\} 
  \displaybreak[0]\\
 &\quad + 
  \sum_{k=1}^{n-1} \sum_{j=1}^{n-k} z_{(a+b)j+b} 
  \left\{ 
  S( z_{a+b}^{n-k-j} ) * 
  \left( \sum_{i=0}^{k} A_{i,k-i} + \sum_{i=1}^{k} z_{a} B_{i-1,k-i} \right)  
  \right\} 
  \displaybreak[0]\\
 &\quad + \sum_{k=0}^{n-1} \sum_{j=1}^{n-k} z_{(a+b)j} 
  \left( S( z_{a+b}^{n-k-j}) * z_{c} (z_{b} z_{a})^{k} z_{b} \right) \displaybreak[0]\\
 &\quad + \sum_{k=0}^{n-1} \sum_{j=1}^{n-k} z_{c} 
  \left( z_{(a+b)j} S( z_{a+b}^{n-k-j}) * (z_{b} z_{a})^{k} z_{b} \right) \displaybreak[0]\\
 &\quad + \sum_{k=0}^{n-1} \sum_{j=1}^{n-k} z_{(a+b)j+c} 
  \left( S( z_{a+b}^{n-k-j}) * (z_{b} z_{a})^{k} z_{b} \right) \displaybreak[0]\\
 &\quad + 
  \sum_{j=1}^{n} z_{(a+b)j} 
  \left( S( z_{a+b}^{n-j}) * z_{b} z_{c}  \right) 
  + \sum_{j=1}^{n} z_{b} 
  \left( z_{(a+b)j} S( z_{a+b}^{n-j}) * z_{c} \right) \displaybreak[0]\\
 &\quad + 
  \sum_{j=1}^{n} z_{(a+b)j+b} 
  \left( S( z_{a+b}^{n-j}) * z_{c} \right)  
  + \sum_{i=0}^{n} z_{b} A_{i,n-i} + \sum_{i=0}^{n} B_{i,n-i} 
  \displaybreak[0]\\
 &=\sum_{k=1}^{n-1} \sum_{j=1}^{n-k} z_{(a+b)j} 
  \left\{ 
  S( z_{a+b}^{n-k-j} ) * 
  \left( \sum_{i=0}^{k} z_{b} A_{i,k-i} + \sum_{i=1}^{k} B_{i,k-i} \right) 
  \right\} 
  \displaybreak[0]\\
 &\quad + \sum_{k=0}^{n-1} \sum_{j=1}^{n-k} z_{(a+b)j} 
  \left( S( z_{a+b}^{n-k-j}) * z_{c} (z_{b} z_{a})^{k} z_{b} \right) 
  + \sum_{j=1}^{n} z_{(a+b)j} 
  \left( S( z_{a+b}^{n-j}) * z_{b} z_{c}  \right) \displaybreak[0]\\
 &\quad + 
  \sum_{k=1}^{n-1} \sum_{j=1}^{n-k} z_{b} 
  \left\{ 
  z_{(a+b)j} S( z_{a+b}^{n-k-j} ) * 
  \left( \sum_{i=0}^{k} A_{i,k-i} + \sum_{i=1}^{k} z_{a} B_{i-1,k-i} \right)  
  \right\} 
  \displaybreak[0]\\ 
 &\quad + 
  \sum_{k=1}^{n-1} \sum_{j=1}^{n-k} z_{(a+b)j+b} 
  \left\{ 
  S( z_{a+b}^{n-k-j} ) * 
  \left( \sum_{i=0}^{k} A_{i,k-i} + \sum_{i=1}^{k} z_{a} B_{i-1,k-i} \right)  
  \right\} 
  \displaybreak[0]\\
 &\quad + \sum_{j=1}^{n} 
  \left\{ 
  z_{b} \left( z_{(a+b)j} S( z_{a+b}^{n-j}) * z_{c} \right) 
  + z_{(a+b)j+b} \left( S( z_{a+b}^{n-j}) * z_{c} \right) 
  \right\} \displaybreak[0]\\
 &\quad + \sum_{k=0}^{n-1} \sum_{j=1}^{n-k} z_{c} 
  \left( z_{(a+b)j} S( z_{a+b}^{n-k-j}) * (z_{b} z_{a})^{k} z_{b} \right)  \\
 &\quad + \sum_{k=0}^{n-1} \sum_{j=1}^{n-k} z_{(a+b)j+c} 
  \left( S( z_{a+b}^{n-k-j}) * (z_{b} z_{a})^{k} z_{b} \right) \displaybreak[0]\\
 &\quad + 
  \sum_{i=0}^{n} z_{b} A_{i,n-i}
  + \sum_{i=0}^{n} B_{i,n-i} 
  \displaybreak[0]\\
 &\stackrel{(\ref{eq:7})}{=}
  \sum_{k=1}^{n-1} \sum_{j=1}^{n-k} z_{(a+b)j} 
  \left\{ 
  S( z_{a+b}^{n-k-j} ) * 
  \left( \sum_{i=0}^{k} z_{b} A_{i,k-i} + \sum_{i=0}^{k} B_{i,k-i} \right) 
  \right\} 
  \displaybreak[0]\\
 &\quad + \sum_{j=1}^{n} 
  \left\{ 
  z_{(a+b)j} \left( S( z_{a+b}^{n-j}) * z_{c} z_{b}  \right) 
  + z_{(a+b)j} \left( S( z_{a+b}^{n-j}) * z_{b} z_{c}  \right) 
  \right\} \displaybreak[0]\\
 &\quad + 
  \sum_{k=1}^{n-1} \sum_{j=0}^{n-k} z_{(a+b)j+b} 
  \left\{ 
  S( z_{a+b}^{n-k-j} ) * 
  \left( 
  \sum_{i=0}^{k} A_{i,k-i} 
  + \sum_{i=1}^{k} z_{a} B_{i-1,k-i} 
  \right)  
  \right\} 
  \displaybreak[0]\\ 
 &\quad + \sum_{j=0}^{n}  z_{(a+b)j+b} \left( S( z_{a+b}^{n-j}) * z_{c} \right)  
  \displaybreak[0]\\
 &\quad + 
  \sum_{k=0}^{n-1} \sum_{j=0}^{n-k} z_{(a+b)j+c} 
  \left( S( z_{a+b}^{n-k-j}) * (z_{b} z_{a})^{k} z_{b} \right) 
  \displaybreak[0]\\
 &\quad + 
  \sum_{i=0}^{n} z_{b} A_{i,n-i} 
  + \sum_{i=0}^{n} B_{i,n-i} 
  \displaybreak[0]\\
 &= \sum_{k=0}^{n-1} \sum_{j=1}^{n-k} z_{(a+b)j} 
  \left\{ 
  S( z_{a+b}^{n-k-j} ) * 
  \left( \sum_{i=0}^{k} z_{b} A_{i,k-i} + \sum_{i=0}^{k} B_{i,k-i} \right) 
  \right\} 
  \displaybreak[0]\\
 &\quad 
  + \sum_{k=0}^{n-1} \sum_{j=0}^{n-k} z_{(a+b)j+b} 
  \left\{ 
  S( z_{a+b}^{n-k-j} ) * 
  \left( \sum_{i=0}^{k} A_{i,k-i} + \sum_{i=1}^{k} z_{a} B_{i-1,k-i} \right)  
  \right\} 
  \displaybreak[0]\\
 &\quad + \sum_{j=0}^{n} z_{(a+b)j+c} \sum_{k=0}^{n-j}  
  S( z_{a+b}^{n-j-k}) * z_{b} (z_{a} z_{b})^{k} - z_{c} (z_{b} z_{a})^{n} z_{b} \displaybreak[0] \\
 &\quad + 
  \sum_{i=0}^{n} z_{b} A_{i,n-i} + \sum_{i=0}^{n} B_{i,n-i} 
  \displaybreak[0]\\
 &\stackrel{(\ref{eq:3})}{=}
  \sum_{k=0}^{n-1} \sum_{j=1}^{n-k} z_{(a+b)j} 
  \left\{ 
  S( z_{a+b}^{n-k-j} ) * 
  \left( \sum_{i=0}^{k} z_{b} A_{i,k-i} + \sum_{i=0}^{k} B_{i,k-i} \right) 
  \right\} 
  \displaybreak[0]\\
 &\quad + 
  \sum_{k=0}^{n-1} \sum_{j=0}^{n-k} z_{(a+b)j+b} 
  \left\{ 
  S( z_{a+b}^{n-k-j} ) * 
  \left( \sum_{i=0}^{k} A_{i,k-i} + \sum_{i=1}^{k} z_{a} B_{i-1,k-i} \right)  
  \right\} 
  \displaybreak[0]\\
 &\quad + \sum_{j=0}^{n} z_{(a+b)j+c} S( z_{b} (z_{a} z_{b})^{n-j} )  
  \displaybreak[0]\\
 &\quad + 
  \sum_{i=0}^{n} z_{b} A_{i,n-i} + \sum_{i=1}^{n} B_{i,n-i} 
  \displaybreak[0]\\
 &=\sum_{j=1}^{n} \sum_{k=0}^{n-j} z_{(a+b)j} 
  \left\{ 
  S( z_{a+b}^{n-k-j} ) * 
  \left( \sum_{i=0}^{k} z_{b} A_{i,k-i} + \sum_{i=0}^{k} B_{i,k-i} \right) 
  \right\} 
  \displaybreak[0]\\
 &\quad + 
  \sum_{j=0}^{n} \sum_{k=0}^{n-j} z_{(a+b)j+b} 
  \left\{ 
  S( z_{a+b}^{n-k-j} ) * 
  \left( \sum_{i=0}^{k} A_{i,k-i} + \sum_{i=1}^{k} z_{a} B_{i-1,k-i} \right)  
  \right\} 
  \displaybreak[0]\\
 &\quad + \sum_{k=0}^{n} z_{(a+b)k+c} S( z_{b} (z_{a} z_{b})^{n-k} ).
\end{align*}
Therefore, we have
\begin{align*}
 \lefteqn{
 2 \sum_{k=0}^{n} z_{(a+b)k+c} * S( z_{b} (z_{a} z_{b})^{n-k} ) 
 }
 \\
&\quad\quad
 - \sum_{k=0}^{n} S( z_{a+b}^{n-k} ) * 
 \left\{ \sum_{i=0}^{k} z_{b} A_{i,k-i} + \sum_{i=0}^{k} B_{i,k-i} \right\} 
 \\ 
 &= \sum_{k=0}^{n} z_{(a+b)k+c} S( z_{b} (z_{a} z_{b})^{n-k} ) \displaybreak[0]\\
 &\quad + 2 \sum_{k=0}^{n-1} \sum_{j=0}^{n-k} z_{(a+b)(k+j)+b+c} S( (z_{a} z_{b})^{n-k-j} ) 
  \displaybreak[0]\\
 &\quad + 2 \sum_{k=0}^{n-1} \sum_{j=1}^{n-k} z_{(a+b)(k+j)+c} S( z_{b} (z_{a} z_{b})^{n-k-j} ) 
  + 2 z_{(a+b)n+b+c} \displaybreak[0]\\ 
 &\quad + \sum_{j=1}^{n} z_{(a+b)j} 
  \left\{ 2 \sum_{k=0}^{n-j} z_{(a+b)k+c} * S( z_{b} (z_{a} z_{b})^{n-j-k}) \right. \\
 &\qquad\qquad\qquad\qquad 
  \left. - \sum_{k=0}^{n-j} S( z_{a+b}^{n-j-k}) * 
  \left( \sum_{i=0}^{k} z_{b} A_{i,k-i} +\sum_{i=0}^{k} B_{i,k-i} \right) 
  \right\} 
  \displaybreak[0]\\
 &\quad + \sum_{j=0}^{n} z_{(a+b)j+b} 
 \left\{ 2 \sum_{k=0}^{n-j} z_{(a+b)k+c} * S( (z_{a} z_{b})^{n-j-k}) \right. \\
 &\qquad\qquad\qquad\qquad 
  \left. - \sum_{k=0}^{n-j} S( z_{a+b}^{n-j-k}) * 
  \left( \sum_{i=0}^{k} A_{i,k-i} + \sum_{i=1}^{k} z_{a} B_{i-1,k-i} \right) 
  \right\} 
  \displaybreak[0]\\
 &= \sum_{k=0}^{n} z_{(a+b)k+c} S( z_{b} (z_{a} z_{b})^{n-k} ) \displaybreak[0]\\ 
 &\quad + 2 \sum_{k=0}^{n-1} \sum_{j=0}^{n-k} z_{(a+b)(k+j)+b+c} S( (z_{a} z_{b})^{n-k-j} ) 
  \displaybreak[0]\\
 &\quad + 2 \sum_{k=0}^{n-1} \sum_{j=1}^{n-k} z_{(a+b)(k+j)+c} S( z_{b} (z_{a} z_{b})^{n-k-j} ) 
  + 2 z_{(a+b)n+b+c} \displaybreak[0]\\
 &\quad + 
  \sum_{j=1}^{n} z_{(a+b)j} 
  \left\{ \sum_{k=0}^{n-j} S( z_{b} A_{k,n-j-k} ) + \sum_{k=0}^{n-j} S( B_{k,n-j-k} ) \right\} 
  \displaybreak[0]\\
 &\quad + 
  \sum_{j=0}^{n} z_{(a+b)j+b} 
  \left\{ \sum_{k=0}^{n-j} S( A_{k,n-j-k}) + \sum_{k=0}^{n-j-1} S(z_{a} B_{k,n-j-1-k} ) \right\} 
  \displaybreak[0]\\
 & \qquad \mathrm{\big( by \; induction \; hypothesis \; and \; (\ref{eq:10}) \; for} \; n \big) 
  \displaybreak[0]\\
 &\stackrel{(\ref{eq:13})}{=} 
  \sum_{k=0}^{n} S( z_{b} A_{k,n-k} ) + \sum_{k=0}^{n} S( B_{k,n-k} ). 
\end{align*}
\end{proof}

\begin{proof}[Proof of Theorem \ref{thmc}]
By (\ref{eq:8}), we obtain
\[\sum_{\vec{s}_{2n}\in I_{2n}} \zeta^{*}(\vec{s}_{2n}) 
=2\sum_{k=0}^{n} \zeta(4k+2) \zeta^{*}(\underbrace{3,1,\ldots,3,1}_{2n-2k})
-\sum_{k=0}^{n} \zeta^{*}(\underbrace{4,\ldots,4}_{n-k}) 
 \sum_{\vec{s}_{2k}\in I_{2k}} \zeta(\vec{s}_{2k}).\]
Hence, we have the assertion by Theorem \ref{thm1}, Theorem \ref{thm3}, 
Theorem \ref{thma} and Theorem \ref{thmb}.
\end{proof}

\section*{Acknowledgements}
The author would like to thank Professor Masanobu Kaneko for many useful advices. 
Also, he wants to thank Kentaro Ihara, Jun Kajikawa and Tatsushi Tanaka for 
helpful comments and suggestions.

\end{document}